\documentclass[12pt]{amsart} 
\usepackage{amssymb}
 \calclayout
\theoremstyle{plain}
\newtheorem{thm}{Theorem}[section]
\newtheorem{cor}{Corollary}[section]

\newtheorem{prop}{Proposition}[section]

\newtheorem{defn}{Definition}[section]
\newtheorem{ex}[thm]{Example}
\newtheorem{rem}{Remark}[section]

\newcommand{\begintheorem}{\addtocounter{equation}{1}\begin{theorem}}
\newcommand{\beginlemma}{\addtocounter{equation}{1}\begin{lemma}}
\newcommand{\beginproposition}{\addtocounter{equation}{1}\begin{proposition}}
\newcommand{\begindefinition}{\addtocounter{equation}{1}\begin{definition}}
\newcommand{\begincorollary}{\addtocounter{equation}{1}\begin{corollary}}

\usepackage[bookmarksnumbered, colorlinks, plainpages]{hyperref}
\hypersetup{colorlinks=true,linkcolor=red, anchorcolor=green, citecolor=cyan, urlcolor=red, filecolor=magenta, pdftoolbar=true}
\begin{document}
\title[ Factorization Theorem through a Dunford-Pettis $ p $-convergent operator]
 {Factorization Theorem through a Dunford-Pettis $ p $-convergent operator }
\author {M. Alikhani.}
\address{Department of Mathematics, University of Isfahan}
\email{m2020alikhani@ yahoo.com}

\begin{abstract} In this article, we introduce the notion of $ p $-$ (DPL) $ sets.\ Also,   a factorization result for  differentiable 
mappings through Dunford-Pettis $ p $-convergent operators is investigated.\ Namely, if $ X ,Y $ are  real Banach spaces and
 $ U $ is an open convex subset of  $ X, $ then
  we obtain that,
given a differentiable mapping $ f : U \rightarrow Y $ its derivative $ f^{\prime} $ takes $ U $-bounded sets
into $ p $-$ (DPL) $ sets if and only if it happens $ f=g\circ S, $ where S is a Dunford-Pettis $ p $-convergent
 operator  from $ X $ into a suitable Banach space $ Z$ and $ g : S(U) \rightarrow Y $  is a G$ \hat{a} $teaux differentiable mapping with some additional properties.
\end{abstract}
\maketitle
\section{Introduction}
Results on factorization through bounded linear operators, polynomials
and holomorphic mappings between Banach spaces have been obtained in
recent years by several authors.\ For instance, M. Gonz\'alez and J. M. Guti\'errez {\rm (\cite[Theorem 16]{gg1})},
proved that a holomorphic mapping $  f : X \rightarrow Y $
between complex Banach spaces is weakly uniformly continuous on bounded subsets
if and only if there exist a Banach space $ Z, $ a compact  operator $ S : X \rightarrow Z, $
and a holomorphic mapping $ g : Z \rightarrow Y $ such that $ f = g \circ S. $\ Recently, Cilia and  Guti\'errez {\rm (\cite[Theorem 2.1]{ce4})},
obtained
a factorization
result for differentiable mappings through compact operators.\ In fact, they
proved that a mapping $ f : X \rightarrow Y $ between real Banach spaces is differentiable
and its derivative $ f^{\prime} $ is a compact mapping with values in the space $ K(X, Y ) $ of
compact operators from $  X$ into $ Y $ if and only if $ f $ may be written in the form
$ f = g \circ S, $ where the intermediate space is normed, $ S $ is a precompact operator,
and $ g $ is a G$ \hat{a} $teaux differentiable mapping with some additional properties.\ 
In the present paper,
we  introduce the notion of $  p$-$ (DPL) $ set  in order to
obtaining a factorization result for a differentiable mapping through
a Dunford-Pettis $ p $-convergent operator.\
\section{ Notions and Definitions }
Throughout this paper $ 1\leq p < \infty $ and $ 1\leq p < q\leq \infty,$ except for the cases where we consider other assumptions.\ Also, we suppose
$ X $ and $ Y$ will denote real Banach spaces, $ U \subseteq X $ will be an  open convex subset,
 $p^{\ast}$ is the H$\ddot{\mathrm{o}}$lder conjugate of $p;$ if $ p=1,~~ \ell_{p^{\ast}} $
plays the role of $ c_{0} .$\ The unit coordinate vector in $ \ell_{p} $ (resp.\ $ c_{0} $ or $\ell_{\infty} $) is denoted by $ e_{n}^{p} $ (resp.\ $ e_{n} $).\ We represent the set of all natural numbers and  the real field by 
 $ \mathbb{N} $ and $ \mathbb{R} $ respectively.\ Given $ x, y \in  X, $ we write $ I(x, y) $ for the
segment with bounds x and y, that is, $ I(x, y)=\lbrace x+\lambda (y-x): 0\leq \lambda \leq 1   \rbrace .$\ For
given an open set $ U \subseteq X,  $ we say that a subset $ B \subset U $ is $ U $-bounded if it is bounded and the distance $ dist(B,  U) $
between B and the boundary $  $ of $  U$ is strictly positive. Clearly, if $ U = X,  $  the $  U$-bounded sets coincide with the bounded
sets.\ A sequence $ (x_{n}) \subset U $ is $ U $-bounded, if the set $ \lbrace x_{n}: n \in  \mathbb{N}\rbrace $ is $ U $-bounded.\ A mapping $ f : U \rightarrow Y $ is compact if it takes $  U$-bounded subsets of $ U $ into relatively compact subsets of $ Y. $\
 For given a mapping $ f : U \rightarrow Y $ and a class $  \mathcal{M}$ of subsets of $ U $ such that every
singleton belongs to $  \mathcal{M},$ the mapping $  f$ is $  \mathcal{M}$-differentiable at $ x \in U, $ if there exists
an operator $ f^{\prime}(x) \in L(X,Y)$ such that
$$\lim_{\varepsilon\rightarrow 0}\frac{f(x+\varepsilon y)-f(x)-f^{\prime}(x)(\varepsilon y)}{\varepsilon} =0 $$
uniformly with respect to y on each member of $ \mathcal{M} .$\ In this case, we write
$ f\in D_{\mathcal{M}}(x, Y) .$\
We say that $ f $ is G$ \hat{a} $teaux differentiable at $  x$ if
$ f\in D_{\mathcal{M}}(x, Y) $ where  $ \mathcal{M} $ is the class of all single-point subsets of $ X. $\ We also, say that $ f $ is Fr\'echet differentiable at $ x $ if $ f\in D_{\mathcal{M}}(x, Y) ,$  where $ \mathcal{M} $ is the
class of all bounded subsets of $ X$ \cite{gg1}.\\
The word ` an operator' will always mean a bounded linear operator.\ For any Banach space $ X, $ the dual space of bounded
linear functionals on $  X$ will be denoted by $ X^{\ast} .$\
 We denote the closed unit ball of $ X$ by $ B_{X} $ and the identity
operator on $ X$ is denoted by $ id_{X}. $\ For a bounded linear operator $ T : X \rightarrow Y, $ the adjoint of the operator $ T $
is denoted by $ T^{\ast}. $\ The space of all bounded linear operators, weakly compact operators from $  X$ to $  Y$ will be denoted by $ L(X,Y ) $ and
$ W(X,Y ) $ respectively. \
 We denote $ C^{k} (X)$ the space of all
real-valued k-times continuously  differentiable functions on $ X. $\\ 
Recall that a bounded linear operator is completely
continuous, if it takes weakly convergent sequences into (norm) convergent sequences.\ The subspace of all such operators
is denoted by $ CC(X, Y ).$\
Let us recall from \cite{djt} that, a sequence $ (x_{n})_{n} $ in $ X $ is said to be weakly $ p $-summable, if $ (x^{\ast}(x_{n}))_{n} \in \ell_{p}$ for each $ x^{\ast}\in X^{\ast}.$\ Note that, a sequence $(x_{n})_{n}$ in $ X $ is said to be weakly $ p $-convergent to $ x\in X,$ if $ (x_{n} - x)_{n}$ is a weakly $ p $-summable sequence in $ X. $\
Castillo \cite{C1}, defined the ideal of $  p$-converging operators,
as those bounded linear operators transforming weakly $ p $-summable sequences into norm
null sequences.\ The class of $ p $-convergent operators from $ X $ into $ Y$  is denoted by $ C_{p}(X,Y) .$\ A Banach space $  X$ has the $ p $-Schur property (in
short $ id_{X}\in (C_{p}) $), if every weakly $  p$-summable sequence in $  X$ is norm null.
It is clear that, $ X $ has the $ \infty $-Schur property if and only if every weakly
null sequence in $  X$ is norm null.\ Hence, the $ \infty $-Schur property coincides with
the Schur property.\ Recall that, a Banach space $ X$ has the Dunford-Pettis property of order $p$ ($X \in (DPP_{p})$), if every weakly compact operator on
$ X $ is $ p $-convergent \cite{cs}.\ 
A bounded
subset $ K$ of $ X $ is  Dunford-Pettis, if
every weakly null sequence $(x^{\ast}_{n})_n $ in $ X^{\ast}, $ converges uniformly to zero on the set $K$ \cite{An}.\\ Throughout this paper,
inspired by Right null and Right Cauchy sequences \cite{ce1, g8}, for convenience
we apply the notions $ p$-Right null and $ p $-Right Cauchy sequences instead of weakly $ p $-summable and weakly $ p $-Cauchy sequences which are Dunford-Pettis sets, respectively.\
The space of all finite regular Borel signed measures on the compact space $ \Omega $ is denoted by $ M(\Omega)=C(\Omega)^{\ast} .$\ 
We refer the reader for undefined terminologies to the
classical references \cite{AlbKal, di1}.
\section{Main results}
In this section, we introduce the notions of $  p$-$ (DPL) $ sets and  the mapping $ p $-Right sequentially continuous.\
Then, we obtain a  factorization result for a differentiable mapping through
a Dunford-Pettis $ p $-convergent operator.\
\begin{defn}\label{d1}
Let $ K \subset L(X, Y ). $\ We say that $ K $ is a $ p $-$ (DPL) $ set, if for every
$ p $-Right null sequence $ (x_{n})_{n} $ in $ X, $ it follows
$$ \lim _{n} \sup_{T \in K} \Vert T(x_{n})\Vert =0.$$
\end{defn}
 Note that  the
$ p $-$ (DPL) $  subsets of $ X^{\ast} $ are more often called $ p $-Right sets introduced by  Ghenciu \cite{g9}.\\
Recall that, an operator $T: X\rightarrow Y $
is said to be Dunford-Pettis $ p$-convergent, if it takes Dunford-Pettis weakly $ p $-summable sequences to norm null sequences.\ The class of Dunford-Pettis $ p $-convergent operators from $ X $ into $ Y $ is denoted by $ DPC_{p}(X,Y) $ \cite{g9}.\\ The following Proposition gives some additional properties of $ p $-$ (DPL) $ sets.\
\begin{prop}\label{p1}
$ \rm{(i)} $ If $ K \subset DPC_{p}(X, Y ) $ is a relatively compact set, then it is a $ p $-$ (DPL) $ set
in $ L(X, Y ). $\\
$ \rm{(ii)} $ Absolutely closed convex hull of a $ p $-$ (DPL) $ set in $ L(X, Y ) $ is $ p $-$ (DPL). $\\
$ \rm{(iii)} $
Every relatively weakly compact subset of $ X^{\ast} $ is a $ p$-$ (DPL) $ set in $ X^{\ast} .$\\
$ \rm{(iv)} $ $ T \in DPC_{p}(X,Y) $ if and only if $ T^{\ast}(B_{Y^{\ast}})$ is a $ p $-$ (DPL) $ set in $ X^{\ast}. $\\
$ \rm{(v)} $ If $1\leq p < q \leq \infty,  $ then every $ q $-$ (DPL) $ subset
of $ X^{\ast}$ is a $ p $-$ (DPL) $ subset of $ X^{\ast} ,$ but the
converse of this assertion is not correct.\\
$ \rm{(vi)} $  If $K \subset L(X, Y ) $ is a $ p$-$ (DPL)$ set, then every $ T\in K $ is a Dunford-Pettis $ p $-convergent operator.\\
$ \rm{(vii)} $ If $  \Omega$ is a compact Hausdorff space, then every $ p $-$ (DPL) $ set in $ M(\Omega) $ is $ q $-$ (DPL) .$
\end{prop}

\begin{rem}\label{r1}
$ \rm{(i)} $ It is clear that every
$ q $-$ (DPL) $ subset of $L(X,Y) $ is 
$ p $-$ (DPL), $ whenever $1\leq p < q \leq \infty. $\ Also,
it is interesting to obtain conditions under which every $ p $-$ (DPL) $ set in
the space $ L(X,Y) $ is $ q $-$ (DPL). $\ In my opinion, this is a very
interesting but, difficult question.\ In particular if $ K\subset X^{\ast} ,$ we answer  this question.\ Indeed,  
we obtain a characterization for those Banach spaces in which $ p $-$ (DPL) $
 sets are  $ q $-$ (DPL) $ (see Definition $ \rm{4.1} $ and  Theorem $ \rm{4.4} $  in \cite{ma}).\\
$ \rm{(ii)} $
 Proposition $ \rm{\ref{p1}} $ assertion $ \rm{(iii)} $  implies that every relatively weakly compact
subset of a topological dual Banach space is $ p $-$ (DPL), $ while the converse of this implication is false.\ For instance, the unit ball of $ \ell_{\infty} $ is a $ p$-$ (DPL) $ set, but it is
not weakly compact.\\
$ \rm{(iii)} $ There is a relatively weakly compact set in $ K(c_{0},c_{0})$ so that is not a $ p$-$ (DPL) $ set.\ In fact, consider
the operator $ T:\ell_{2}\rightarrow K(c_{0},c_{0})$
given by $ T(\alpha)(x) =(\alpha_{n}x_{n}), ~~~~ \alpha=(\alpha_{n})\in \ell_{2}, ~~~~ x=(x_{n}) \in c_{0}.$\ It is clear that $ T(B_{\ell_{2}}) $ is relatively
weakly compact, since $ T(e^{2}_{n})(e_{n}) =e_{n}.$\ But it is not a $ p$-$ (DPL) $ set in $ K(c_{0} ,c_{0}). $
\end{rem}
 
\begin{defn}\label{d2}
We say that  the mapping $ f :U\rightarrow Y$ is $ p $-Right sequentially continuous or $ \tau_{r} $ sequentially continuous of order $ p,$ if it takes $ p $-Right Cauchy $ U $-bounded
sequences of $ U $ into norm convergent sequences in $ Y. $
\end{defn}
Note that the mapping $ f :U\rightarrow Y$ is $ \infty $-Right sequentially continuous or Right sequentially continuous, if it takes Right Cauchy and $ U $-bounded
sequences of $ U $ into  norm convergent sequences in $ Y. $
\begin{rem}\label{r2}
If $ 1\leq p<q\leq \infty, $ then the $ q $-Right sequentially continuous maps are precisely the $ p $-Right sequentially continuous.\ 
We do not have any example of a mapping $ p $-Right sequentially continuous
  which is not $ q $-Right sequentially continuous.\ Hence,
it would be interesting to obtain conditions under which every $ p $-Right sequentially continuous map is  $ q $-Right sequentially continuous.\ In my opinion, this is a very
interesting, but its difficult question?

\end{rem}
\begin{prop}\label{p2}
Let $ U \subseteq X $ be an open convex subset, and let $ f : U \rightarrow Y $ be  a
differentiable mapping such that $ f^{^{\prime}} : U \rightarrow DPC_{p}(X, Y ) $ is $ \tau_{r}$-sequentially continuous on
$ U $-bounded sets.\ Then, $ f^{\prime} $ takes Dunford-Pettis and $ U$-bounded sets into $ p $-$ (DPL) $ sets.
\end{prop}
\begin{proof}
Let $ K $ be a $ U$-bounded and Dunford-Pettis set.\ It is well known that, $ K$ is a Rosenthal
set (see,{\rm (\cite[Corollary 17]{g12})}).\ So, by the hypothesis, $ f^{\prime} (K)$ is relatively compact in $ DPC_{p}(X,Y). $\ Hence,
by the part $ \rm{(i)} $ of Proposition \ref{p1}, $ f^{\prime} (K)$ is a $ p $-$ (DPL) $ set.
\end{proof}

\begin{thm}\label{t1}
Let $ U \subseteq X $ be an open convex subset, and let $ f : U \rightarrow Y $ be a
differentiable mapping such that for every $ U $-bounded Dunford-Pettis set $ K, f^{\prime}(K) $ is a $ p $-$ (DPL) $ set in $ L(X, Y ). $\ Then, $ f$ is $ p $-Right sequentially continuous.
\end{thm}
\begin{proof}
Let $ (x_{n})_{n} $ be a $ U $-bounded $ p $-Right Cauchy sequence.\ By the Mean Value Theorem, for all $ n,m\in \mathbb{N} ,$ there is
$ c_{i,j} \in I(x_{n},x_{m})$ such that
$$ \parallel f(x_{n}) -f(x_{m}) \parallel \leq \parallel f^{\prime}(c_{i,j})(x_{n}-x_{m}) \parallel \leq \sup_{i,j} \parallel f^{\prime}(c_{i,j})(x_{n}-x_{m}) \parallel$$
It is clear that, the set $ K:=\lbrace c_{i,j}:i,j \in \mathbb{N} \rbrace $ is contained in the convex hull of all $ x_{n} $ and then in
$ U, $ since $ U $ is a convex set.\ Moreover $ K$ is still a Dunford-Pettis and $ U $-bounded set.
By the hypothesis, $ f^{\prime}(K) $ is a $ p $-$ (DPL) $ set in $ L(X,Y) $.\ Since $ (x_{n}-x_{m}) $ is a $ p $-Right null
sequence, it follows that
$ \displaystyle\lim_{_{n,m}}\sup_{i,j}\parallel f^{\prime}(c_{i,j})(x_{n}-x_{m}) \parallel =0. $
Therefore, $ \Vert f(x_{n})-f(x_{m}) \Vert\rightarrow 0. $
\end{proof}
\begin{ex}\label{e1}
Let $ h\in C^{1}(\mathbb{R}) .$\ Define $ f:c_{0}\rightarrow \mathbb{R} $ by 
$ f((x_{n})_{n}) = \displaystyle\sum _{n=1}^{\infty}\frac{h(x_{n})}{2^{n}}.$\
By using the same argument as in the {\rm (\cite[Example 2.4]{ce4})}, one can show that $ f $ is differentiable such that  
$ f^{\prime} ((x_{n})_{n})=(\frac{h^{\prime}(x_{n})}{2^{n}})_{n}\in \ell_{1}.$\  It is easy to verify that
 $ f^{\prime}:c_{0} \rightarrow L(c_{0},\mathbb{R})=DPC_{p}(c_{0},\mathbb{R})$ is compact.\ So, $ f^{\prime}(B_{c_{0}}) $ is a relatively compact set in $ DPC_{p} (c_{0},\mathbb{R}).$\ Hence  the part $ \rm{(i)} $ of Proposition $ \rm{\ref{p1}} $, implies that $ f^{\prime}(B_{c_{0}}) $ is a $ p $-$ (DPL) $ set.\ 
Now, let $ K $ be an arbitrary $ U $-bounded and Dunford-Pettis set in $ B_{c_{0}}. $\  Clearly,
 $ f^{\prime}(K)$ is a $ p $-$ (DPL) $ set in $ L(c_{0}, \mathbb{R}) .$\ Hence, 
the Theorem $ \rm{ \ref{t1}} $, shows that $ f$ is $ p $-Right sequentially continuous.
\end{ex}
Let us recall from \cite{ccl1}, that a bounded subset $ K $ of $ X $ is a $ p $-$ (V^{\ast} ) $ set, if $ \displaystyle \lim_{n\rightarrow\infty}\displaystyle \sup_{x\in K}\vert x^{\ast}_{n}(x) \vert =0,$
for every weakly $ p $-summable sequence $ (x^{\ast}_{n})_{n} $ in $ X^{\ast}.$

\begin{prop}\label{p3}
Suppose that $ 2<p \leq \infty$ and $ K \subset L(X, Y ) $ is a $ p $-$ (DPL) $ set.\ If $ S \in L(G,X) $ is a bounded linear
operator with $ p $-convergent adjoint, then the set
$ \lbrace S^{\ast}\circ T^{\ast}(B_{Y^{\ast}}) : T\in K \rbrace$
is relatively compact in $ G^{\ast}. $
\end{prop}
\begin{proof}
Take a $ p $-Right null sequence $ (x_{n})_{n} $ in $ X. $\ Then, since $K$ is a
$ p $-$ (DPL) $ set, it follows
$$ \vert \langle x_{n} , T^{\ast}(y^{\ast}) \rangle \vert\leq \vert \langle T(x_{n}) , y^{\ast} \rangle \vert\leq \parallel T(x_{n}) \parallel\rightarrow 0$$
uniformly for $ T\in K $ and $ y^{\ast}\in B_{Y^{\ast}}. $\ So,
$ \lbrace T^{\ast}(B_{Y^{\ast}}): T \in K \rbrace $
is a $ p $-$ (DPL) $ set.\ Adapting of {\rm (\cite[Proposition 3.5]{gg})},
there are a Banach space Z and an operator $ L, $ that takes $ p $-Right Cauchy sequences into norm convergent sequences, such that
$$ \lbrace T^{\ast}(B_{Y^{\ast}}): T \in K \rbrace \subset L^{\ast}(B_{Z^{\ast}}) .$$
Therefore, we have
$$\lbrace (S^{\ast}\circ T^{\ast})(B_{Y^{\ast}}): T \in K \rbrace =S^{\ast}(\lbrace T^{\ast}(B_{Y^{\ast}}): T \in K \rbrace)\subset S^{\ast}(L^{\ast}(B_{Z^{\ast}})).$$
Since $ S^{\ast} $ is $ p $-convergent, {\rm (\cite[Theorem 14]{g12})}, implies that $ S(B_{G}) $ is a $ p $-$ (V^{\ast}) $ set and so, it is a Rosenthal
set (see,{\rm (\cite[Corollary 18]{g12})}).\ Therefore, $ L\circ S $ is compact.\ Hence, $ (S^{\ast }\circ L^{\ast}) $ is compact and we are done.
\end{proof}
\begin{cor}\label{c1}
Suppose that $ 2<p \leq \infty$ and $ K \subset L(X, Y ) $ is a $ p $-$ (DPL) $ set.\ If $ X^{\ast} $ has the $ p $-Schur property,
 then the set
$ \lbrace  T^{\ast}(B_{Y^{\ast}}) : T\in K \rbrace$
is relatively compact in $ X^{\ast}. $

\end{cor}

The space of all  differentiable
mappings $ f : U \rightarrow Y $ whose derivative $ f^{\prime} : U \rightarrow L(X, Y ) $ is uniformly continuous on $ U$-bounded subsets of $ U,$ represented by $ C^{1u}(U,Y) $ \cite{gg1}.
\begin{prop}\label{p4}
Let $ X $ be a Banach space and let $ U $ be an open convex subset of $ X. $\ If for every Banach space $ Y, $ every mapping $ f \in C^{1u}(U,Y) $ whose derivative $ f^{\prime} $
takes $ U $-bounded sets into $ p $-$ (DPL) $ sets, is $ p $-convergent, then $ X $ has the $ (DPP_{p}). $
\end{prop}
\begin{proof}
Let $ T:X\rightarrow c_{0} $ be a weakly compact operator.\ It is easy to verify that $ T$ is Dunford-Pettis $ p $-convergent.\
We proved that $ T $ is $ p $-convergent.\ Since
$$ T^{\prime}(x) =T, ~~~ \forall x\in X,$$
 for every $ U $-bounded set $ B $ and  for every $ p $-Right null sequence $ (x_{n})_{n}, $ it follows
$$ \lim_{n\rightarrow \infty}\sup_{x \in B}\parallel T^{\prime}(x) (x_{n}) \parallel=\lim_{n\rightarrow \infty}\parallel T(x_{n}) \parallel =0.$$
So, $ T^{\prime} $ takes $ U $-bounded sets into $ p $-$ (DPL) $ sets.\ Hence, by the
hypothesis, $ T $ is $ p $-convergent, that is $ X $ has the $ (DPP_{p}). $
\end{proof}
\begin{prop}\label{p5}
Suppose that $ X $ is a Banach space, and $ U$ is an open convex subset of $ X. $\ If for every Banach space $ Y, $ for every mapping $ f \in C^{1u} (U,Y),$ such that $ f^{\prime} $ takes
$ U $ bounded sets into $ p $-$ (DPL) $ sets, it follows that $ f^{\prime} $ is compact and takes its values
into $ K(X, Y ), $ then $ X^{\ast}$ has the Schur property.
\end{prop}
\begin{proof}
Let $ Y$ be an arbitrary Banach space, and let $ T:X\rightarrow Y $ be a weakly
compact operator.\ By the same reasoning as in the proof of Proposition \ref{p4}, $ T^{\prime} $ takes
$ U$-bounded sets into $ p $-$ (DPL) $ sets.\ Hence, $ T^{\prime} $ is compact and takes its
values in $ K(X, Y ). $\ It follows that
$ W(X, Y ) = K(X,Y ) $ for every Banach space $ Y .$\ So, $ X^{\ast}$ has the Schur property property {\rm (\cite[Theorem 2.1]{MZ})}.
\end{proof}

\begin{thm}\label{t2}
Suppose that $ U $ is an open convex subset of $ X$ and $ f:U\rightarrow Y $ is a mapping.\
Then the following assertions are equivalent:\\
$ \rm{(a)} $ $ f $ is differentiable, $ f^{\prime} $ takes $ U $-bounded sets into $ p$-$ (DPL) $ sets and $ f$ is $ p $-Right sequentially continuous.\\
$ \rm{(b)} $ There exist a Banach space $ Z,$ an operator $ S\in DPC_{p}(X, Z ) $ and a mapping
$ g:S(U) \rightarrow Y$ such that:\\
$ \rm{ (i)} $ $ f(x) = g(S(x)) $ for all $ x \in U.$\\
$ \rm{ (ii)} $ $ g \in D_{\mathcal{M}}(S(x), Y ) $ for every $ x \in U, $ where
\begin{center}
$ \mathcal{M}:=\lbrace S(B) :B $ is a $ U $-bounded subset of $ X \rbrace .$
\end{center}\
$ \rm{ (iii)} $ $ g^{\prime} $ is bounded on $ S(B) $ for every $ U $-bounded subset $ B\subset X.$\\
Moreover, if this factorization holds, $ f $ is $p $-Right sequentially continuous.
\end{thm}
\begin{proof}
(a) $ \Rightarrow (b) $ For every $ n\in \mathbb{N}, $ put
$$ W_{n}=\lbrace x \in U: d(x, \partial U)>\frac{1}{n} \rbrace \bigcap nB_{X} .$$
By the hypothesis for every $ r\in \mathbb{N}, $ $ f^{\prime}(\frac{W_{r}}{r \parallel f^{\prime}\parallel_{W_{r}}}) $ is a $ p$-$ (DPL) $ set.\ Now, we set $ K:=\displaystyle\bigcup _{r=1}^{\infty} \frac{f^{\prime}(W_{r})}{r\parallel f^{\prime}\parallel_{W_{r}}}.$\ We claim that, $ K $ is a $ p $-$ (DPL) $ set.\  Indeed, for every $ N\in \mathbb{N}, $ we define $ A_{N} := \displaystyle\bigcup_{r\leq N}\frac{f^{\prime}(W_{r})}{r \parallel  f^{\prime} \parallel}_{W_{r}}$ and $ B_{N} := \displaystyle\bigcup_{r> N}\frac{f^{\prime}(W_{r})}{r \parallel  f^{\prime} \parallel}_{W_{r}}.$\ Since $ K=A_{N}\cup B_{N} $ and $ A_{N} $ is a $ p $-$ (DPL) $ set, it is enough to show that
$ B_{N} $ is a $ p $-$ (DPL) $ set.\  For this purpose, let $ (x_{n})_{n} $ be a $ p $-Right null sequence in $ X $ and $ M= \displaystyle\sup _{n} \parallel x_{n}  \parallel.$\  If $ T \in B_{N} ,$ then there are  $ r>N $ and $ x\in W_{r} $ so that $ T=\frac{f^{\prime}(x)}{r\parallel f^{\prime}  \parallel_{W_{r}}} .$\ It is clear that, $ \parallel T\parallel\leq\frac{1}{r} <\frac{1}{N}.$\ Hence, for each $ N\in \mathbb{N} $ we have: 
$$ \lim_{n\rightarrow \infty}\sup_{T\in K}\parallel T(x_{n})   \parallel =max\lbrace  \lim_{n\rightarrow \infty}\sup_{T \in A_{N}} \parallel T(x_{n})\parallel ,\lim_{n\rightarrow \infty}\sup_{T \in B_{N}} \parallel T(x_{n})\parallel     \rbrace\leq \frac{M}{N} .$$
Therefore, $ \displaystyle \lim_{n\rightarrow \infty}\sup_{T\in K}\parallel T(x_{n})   \parallel =0 $ and so, $ K $ is a $ p $-$ (DPL) $ set.
Now,
as in the proof of {\rm (\cite[Theorem 2.1]{ce3})}, let
$$ V_{K}:=\lbrace x\in X: \displaystyle\sup _{\phi\in K}\parallel \phi(x) \parallel _{Y}=0 \rbrace $$
and $ G:=\frac{X}{V_{k}} .$\ If $ S:X\rightarrow G $ is the quotient map $ G, $ then $ G $ is a normed space
respect the norm $ \parallel S(x) \parallel =\displaystyle\sup _{\phi\in K}\parallel \phi(x) \parallel _{Y}, ~~~\forall x \in X.$\ Suppose that Z is the completion of $ G .$\ If $ (x_{n}) _{n}$ is a $ p $-Right null
sequence in X, then since $ K$ is a $ p $-(DPL) set, $ \parallel S(x_{n}) \parallel =\displaystyle\sup _{\phi\in K}\parallel \phi(x_{n}) \parallel\rightarrow 0. $\ Hence $ S\in DPC_{p} (X,Z).$\
Now we define $ g : S(U) \rightarrow Y $ by $ g(S(x))=f(x) ,~~~~ x\in U.$\ In the first, we proved that g is well defined.\ Suppose that $ \parallel S(x-y) \parallel=0. $\ Since the span of $ K$ contains $ f^{\prime}(U) ,$ we have
$$ \parallel f^{\prime}(c) (x-y) \parallel =0 ~~~ ~~~~~(c \in U).$$
By using the Mean Value Theorem {\rm (\cite[Theorem 6.4]{ch})},
$$\displaystyle\sup _{c\in I(x,y)}\parallel f(x)- f(y) \parallel\leq \parallel f^{\prime}(c) (x-y) \parallel =0,$$ and so $ f(x)=f(y) .$\ Therefore $g$ well defined.\ Now, we show that $ g $ is G$ \hat{a} $teaux differentiable.\ For given $ x,y \in U, $
\begin{equation}\label{eqeq}
\displaystyle\lim_{t\rightarrow 0}\frac{g(S(x)+tS(y))-g(S(x))}{t}=\lim_{t\rightarrow 0}\frac{f(x+ty)-f(x)}{t} =f^{\prime}(x)(y)
\end{equation}
where $\mid t \mid $ is sufficiently small so that $ x + ty \in U. $\ For $ x\in U $ fixed, the mapping
$ g^{\prime}(S(x)) :G\rightarrow Y$ given by $g^{\prime}(S(x)) (S(y)) =f^{\prime}(x)(y) ~~~~~(y \in X) $
is linear.\ Choosing $ r \in \mathbb{N} $ so that $ x \in W_{r}, $ we have
\begin{eqnarray*}
 \parallel g^{\prime}(S(x)) (S(y)) \parallel &=&\\ \parallel f^{\prime}(x)(y) \parallel
&\leq & r\parallel f^{\prime}\parallel_{W_{r}} \parallel S(y) \parallel.
\end{eqnarray*}
Hence $g^{\prime} (S(x)) $ is continuous and may be extended to the completion $ Z$ of $ G. $\ Hence $ g$ is
G$ \hat{a} $teaux differentiable.\ Moreover, since $ f $ is differentiable, for every $ U $-bounded set
$ B, $ the limit in (\ref{eqeq}) exists uniformly with respect to $ S(y) $ in $ S(B). $\ Therefore
$ g \in D_{\mathcal{M}}(S(x), Y ) $ for every $ x \in U, $ where
$$ \mathcal{M}=\lbrace S(B) : B ~  is  ~ a~ U-bounded ~ subset ~ of ~X \rbrace $$
and (ii) is proved.\\
From the inequality $\rm{ (ii)}, $ we have
\begin{center}
$ \parallel g^{\prime}(S(x))\parallel =\displaystyle\sup_{\parallel S(y)\parallel \leq1} \parallel g^{\prime} (S(x)) (S(y)) \parallel \leq r\parallel f^{\prime}\parallel_{W_{r}}, ~~(x \in W_{r})$
\end{center}
and this implies (iii).\\
(b) $ \Rightarrow $ (a). Suppose that there are a Banach space $ Z$ and a Dunford-Pettis $ p $-convergent $ S: X\rightarrow Z, $
and a mapping $ g : S(U) \rightarrow Y $ satisfying $ (b). $\ It is clear that, $ f$ is differentiable.\ We claim that $ f^{\prime} $ takes $ U $-bounded sets into $ p $-$(DPL) $ sets.\ For this purpose, suppose that  $ B $ is a
$ U$-bounded set  and  $ (x_{n})_{n}$ is a 
 $ p $-Right null sequence  in $ X.$\ Since $ S \in DPC_{p}(X,Z), $ we have
\begin{eqnarray*}
 \displaystyle\sup _{x\in B} \parallel f^{\prime}(x)(x_{n})\parallel = \displaystyle \sup_{x\in B}\parallel g^{\prime}(S(x))(S(x_{n}))\parallel
 &\leq &\\ \displaystyle \sup_{x\in B}\parallel g^{\prime}(S(x))\parallel \parallel S(x_{n})\parallel\rightarrow 0.
\end{eqnarray*}
 Therefore $ f^{\prime}(B) $ is a $ p $-$(DPL) $ set.\
Given $ r\in \mathbb{N},$ the mapping $ g $ is uniformly continuous on $ S(W_{r}) .$ Indeed, for $ x, y\in W_{r} ,$
we have
\begin{flushleft}
$ \parallel g(S(x))-g(S(y)) \parallel=\parallel f(x) -f(y) \parallel \leq \displaystyle \sup_{c\in W_{r}}\parallel f^{\prime}(c)(x-y) \parallel $
\end{flushleft}
\begin{center}
$ \leq r\parallel f^{\prime}\parallel_{W_{r}} \displaystyle \sup_{\phi\in K }\parallel \phi(x-y)\parallel=r\parallel f^{\prime}\parallel_{W_{r}} \parallel S(x-y)\parallel $
\end{center}
where we have used that $ W_{r} $ is a convex set.\ Hence, if $ (x_{n})_{n} $ is a $ U $-bounded and
$ p $-Right Cauchy sequence, then the sequence $ (S(x_{n}))$ in $S(W_{r}), $ for a suitable index
$ r, $ is norm Cauchy and therefore, $ (f(x_{n})) = (g(S(x_{n}))) $ is also a norm
Cauchy sequence.\ Hence, $ f$ is $ p $-Right sequentially continuous.
\end{proof}


\begin{thebibliography}{99}
\bibitem{AlbKal} F.\ Albiac and N.\ J.\ Kalton,{ Topics in Banach Space Theory,} Graduate
Texts in Mathematics, 233, Springer, New York, 2006.\
\bibitem{ma} {M.\ Alikhani} Sequentially Right-Like Properties on Banach Spaces, Filomat 33:14 (2019), 4461-4474.\
\bibitem{An} {K.\ T.\ Andrews,} Dunford-Pettis sets in the space of Bochner integrable functions, Math.\ Ann.\ \textbf{241} (1979), 35-41.\


\bibitem{C1} { J.\ M.\ F.\ Castillo,}
$ p $-converging and weakly $ p $-compact operators in $ L_{p} $-spaces.
Actas II Congreso de Analisis Funcional, 1990, Extracta Math.\ volum dedicated
to the II Congress in Functinal Analysis held in Jarandilla de la Vera, Caceres,
June 46-54.
\bibitem{cs} J.\ M.\ F.\ Castillo and F.\ S$\acute{a} $nchez,  Dunford-Pettis-like properties of continuous function vector spaces, Rev.\ Mat.\ Univ.\ Complut.\ Madrid \textbf{6} (1993), 43-59.\
\bibitem{ch}
S.\ B.\ Chae, Holomorphy and Calculus in Normed Spaces, Monogr.\ Textbooks Pure Appl.\ Math.\ 92, Dekker, New York 1985.

\bibitem{ccl} D.\ Chen, J.\ Alejandro Ch$ \acute{a} $vez-Dom$ \acute{i}$nguez and L.\ Li, $ p $-converging operators and Dunford-Pettis
property of order $ p, $ J.\ Math.\ Anal.\ Appl.\ \textbf{461} (2018), 1053-1066.
\bibitem{ce2}
R.\ Cilia and J.\ M.\  Guti\'errez,  Weakly sequentially continuous differentiable mappings, J.\ Math.\ Anal.\ Appl.\ \textbf{360} (2009), 609-623.
\bibitem{ce3}
R.\ Cilia and J.\ M.\  Guti\'errez,  Factorization of weakly continuous differentiable mappings, Bull.\ Braz.\ Math.\ Soc.\ (N.S.) \textbf{40} (2009), 371-380.
\bibitem{ce4}
R.\ Cilia, J.\ M.\ Guti\'errez and G.\ Saluzzo,  Compact factorization of differentiable mappings, Proc.\ Amer.\ Math.\ Soc.\ \textbf{137} (2009), 1743-1752.

\bibitem{ce1} { R.\ Cilia and G.\ Emmanuele,} some isomorphic properties in $ K(X,Y) $ and in projective tensor products,\ Colloq.\ Math.\ \textbf{146} (2017), 239-252.

\bibitem{di1} { J.\ Diestel,}  Sequences and Series in Banach Spaces, Graduate Texts in Mathematics,
Vol. 922, Springer-Verlag, New York, 1984.
\bibitem{djt} { J.\ Diestel, H.\ Jarchow and A.\ Tonge,} Absolutely Summing Operators, Cambridge Univ.\ Press, (1995).\

\bibitem{g8} {I.\ Ghenciu,} A note on some isomorphic properties in projective tensor products, Extracta Math.\ \textbf{32} (2017), 1-24.\
\bibitem{g12} I.\ Ghenciu, The $ p $-Gelfand-Phillips property in spaces of operators and Dunford-Pettis like sets, Acta Math.\ Hungar. \textbf{155} (2018), 439-457.
\bibitem{g9} {I.\ Ghenciu,} Some classes of Banach spaces and complemented subspaces of operators, Adv.\ Oper.\ Theory. \textbf{4} (2019), 369-387.
\bibitem{gg}
M.\ Gonzalez and J.\ M.\ Guti\'errez, The compact weak topology on a Banach space, Proc.\ Roy.\ Soc.\ Edinburgh Sect. A \textbf{120} (1992), 367-379.

\bibitem{gg1} M.\ Gonzalez and J.\ M.\ Guti\'errez,  Factorization of weakly continuous holomorphic mappings,
Studia Math. \textbf{118} (1996), 117-133.

\bibitem{ccl1} {L.\ Li, D.\ Chen and J.\ Alejandro Ch$ \acute{a} $vez-Dom$ \acute{i} $nguez,} Pelczy$ \acute{n} $ski's property $ (V^{\ast}) $ of order $ p$ and its quantification, Math.\ Nachr.\ \textbf{291} (2018) 420-442.

\bibitem{MZ} {S.\ M.\ Moshtaghioun and J.\ Zafarani,} Weak sequential convergence in the dual of operator ideals, J.\ Operator Theory, \textbf{49} (2003), 143-151.

\bibitem {wc} Y.\ Wen and J.\ Chen, Characterizations of Banach spaces with relatively compact Dunford-Pettis sets, Adv. in Math. (China) \textbf{45}(2016), 122-132.

\end{thebibliography}
\end{document}